\documentclass[12pt]{article}

\def\cald{{\mathcal{D}}}

\def\call{{\mathcal{L}}}

\def\calt{{\mathcal{T}}}

\def\calx{{\mathcal{X}}}
\def\({\left(}
\def\){\right)}

\def\vsp{\vspace*{1,5mm}\\ }

\def\bk{\bigskip }
\def\mk{\medskip }
\def\sk{\smallskip }

\def\n{\noindent }
\def\dd{\displaystyle}

\def\barr{\begin{array}}
	\def\earr{\end{array}}

\def\FP{Fokker--Planck}

\usepackage{amsmath, amsthm, amscd, amsfonts, amssymb}

\newtheorem{theorem}{Theorem}[section]

\newtheorem{lemma}[theorem]{Lemma}
\theoremstyle{definition}

\newtheorem{remark}[theorem]{Remark}

\def\1{^{-1}}
\def\vsp{\vspace*{2mm}\\ }

\def\calf{{\mathcal{F}}}
\def\calo{{\mathcal{O}}}
\def\calp{{\mathcal{P}}}
\def\calx{{\mathcal{X}}}

\def\rr{{\mathbb{R}}}
\def\rrd{{\mathbb{R}^d}}
\def\nn{{\mathbb{N}}}
\def\9{{\infty}}

\def\lbb{{\lambda}}

\def\ov{\overline}
\def\vf{{\varphi}}
\def\oo{{\omega}}
\def\ooo{{\Omega}}
\def\pp{{\partial}}

\def\vp{{\varepsilon}}

\def\dd{\displaystyle}
\def\bk{\bigskip }
\def\sk{\smallskip}
\def\n{\noindent }

\def\vsp{\vspace*{2mm}\\ }
\def\ff{\forall }
\def\({\left(}
\def\){\right)}
\def\<{\left<}
\def\>{\right>}
\def\divv{{\rm div}}

\title{Nonlinear Fokker--Planck equations\\ as smooth Hilbertian gradient flows} 
\author{Viorel Barbu\thanks{Al.I. Cuza University and Octav Mayer Institute of Mathematics of  Romanian Academy, Ia\c si, Romania.  Email: vbarbu41@gmail.com} \and Michael Röckner\footnote{Faculty of Mathematics, Bielefeld University, 33615 Bielefeld, Germany. E-mail: roeckner@math.uni-bielefeld.de}}
\date{\textit{In memory of Giuseppe Da Prato}}

\begin{document}
	\maketitle
	\begin{abstract}
	\n Under suitable assumptions on $\beta:\rr\!\to\!\rr, \,D:\rrd\!\to\!\rrd$ and $b:\rrd\!\to\!\rr$,\break the nonlinear \FP\ equation $u_t-\Delta\beta(u)+\divv(Db(u)u)=0$, in $(0,\9)\times\rrd$ where $D=-\nabla\Phi$, 
	 can be identified as a smooth gra\-dient flow $\frac{d^+}{dt}\,u(t)+\nabla E_{u(t)}=0$, $\ff t>0$. Here, $E:\calp^*\cap L^\9(\rrd)\to\rr$ is the energy function associated to the equation, where $\calp^*$ is a certain convex subset of the space of probability densities. $\calp^*$ is invariant under the flow and $\nabla E_u$ is the gradient of $E$, that is, the tangent vector field to $\calp$ at $u$ defined by $\<\nabla E_u,z_u\>_u={\rm diff}\,E_u\cdot z_u$ for all vector fields $z_u$ on $\calp^*$, where $\<\cdot,\cdot\>_u$ is a scalar product on a suitable tangent space $\mathcal{T}_u(\calp^*)\subset\cald'(\rrd)$.\sk
	\end{abstract}

	\section{Introduction}\label{s1}
	We are concerned here with the nonlinear \FP\ equation (NFPE)	
	\begin{equation}\label{e1.1}
	\barr{l}
	u_t-\Delta\beta(u)+{\rm div}(Db(u)u)=0\mbox{ in }(0,\9)\times\rr^d,\\
	u(0,x)=u_0(x),\ x\in\rrd,\vsp
	\earr\end{equation}where $\beta:\rr\to\rr,\ D:\rrd\to\rrd$, $d\ge1,$  and $b:\rr\to\rr$ are assumed to satisfy the following hypotheses 
	\begin{itemize}
		\item[\rm(i)] $\beta\in C^1(\rr)$, $\beta(0)=0,\ 0<\gamma_1\le\beta'(r)\le\gamma_2<\9$, $\ff r\in\rr$.
		\item[\rm(ii)] $b\in C_b(\rr)\cap C^1(\rr)$ and $b(r)\ge b_0>0,$ $|b'(r)r+b(r)|\le\gamma_3<\9,\ \ff r\in\rr.$
		\item[\rm(iii)] $D\in L^\9(\rrd;\rrd)\cap W^{1,1}_{\rm loc}(\rrd;\rrd)$ and  
		$\divv\ D\in L^2(\rrd)+L^\9(\rrd).$ 
		\item[\rm(iv)] $D=-\nabla\Phi$, where $\Phi\in C(\rrd)\cap W^{2,1}_{\rm loc}(\rrd),$ $\Phi\ge1$, $\lim\limits_{|x|\to\9}\Phi(x)=+\9,$ $\Phi^{-m}\in L^1(\rrd)$ for some $m\ge2$.					
	\end{itemize}
NFPE \eqref{e1.1} is modeling  the so called {\it anomalous diffusion} in statistical physics (see, e.g., \cite{7}) and also describes the dynamics of It\^o stochastic processes in terms of their probability densities. In fact, if $u$ is a distributional solution to \eqref{e1.1}, such that $t\to u(t)dx$ is weakly continuous  and $u(t)\in\calp$, $\ff t\ge0$, then there is a probabilistically weak solution $X_t$ to the McKean--Vlasov stochastic differential equation 
\begin{equation}\label{e1.2}
dX_t=D(X_t)b(u(t,X_t))dt+
\(\dd\frac{2\beta(u(t,X_t))}{u(t,X_t)}\)^{\frac12} dW_t,\\
\end{equation}on a probability space $(\ooo,\calf,\mathbb{P},W_t)$ with normal filtration $(\calf_t)_{t\ge0}$. More exactly, one has  $\call_{X_t}\equiv u(t,x)$, where  $\call_{X_t}$ is the density of the marginal law $\mathbb{P}\circ X^{-1}_t$ of $X_t$ with respect to the Lebesgue measure (see \cite{3}, \cite{5b}).  

The function $u:[0,\9)\times\rrd\to\rr$ is called a {\it mild solution} to \eqref{e1.1} if it is $L^1$-continuous, that is $u\in C([0,\9);L^1(\rrd))$, and 
\begin{equation}\label{e1.3}
u(t)=\lim_{h\to0}u_h(t)\mbox{ in }L^1(\rrd),\ \ff t\ge0\end{equation}where, for each $T>0$, $u_h:(0,T)\to L^1(\rrd)$ is defined by
\begin{equation}\label{e1.4}
\barr{rcl}
u_h(t)&=&u^j_h,\ t\in[jh,(j+1)h),\ j=0,1,...,\mbox{$\left[\frac Th\right]$},\vsp
u^{j+1}_h+hAu^{j+1}_h&=&u^j_h,\ j=0,1,...,\mbox{$\left[\frac Th\right]$};\ u^0_h=u_0.\earr\end{equation}Here, $A:L^1(\rrd)\to L^1(\rrd)$ is the operator
\begin{equation}\label{e1.5}
\barr{rcl}
Ay&=&-\Delta\beta(y)+\divv(Db(y)y)\mbox{ in }\cald'(\rrd);\ y\in D(A),\vsp
D(A)&=&\{y\in L^1(\rrd);\ -\Delta\beta(y)+\divv(Db(y)y)\in L^1(\rrd)\}.\earr\end{equation}As shown in \cite{5} (see also \cite{3}--\cite{4}, \cite{5b}), under the above hypotheses (as a matter of fact, for less restrictive assumptions), the domain $D(A)$ is dense in $L^1(\rrd)$, that is,  $\ov{D(A)}=L^1(\rrd)$, and the operator $A$ is $m$-accretive in $L^1(\rrd)$, which means that (see, e.g., \cite{1}, \cite{2}) 

$$\barr{c}
R(I+\lbb A)=L^1(\rrd),\ \ff\lbb>0,\vsp 
\|(I+\lbb A)\1 y_1-(I+\lbb A)\1 y_2\|_{L^1(\rrd)}\le\|y_1-y_2\|_{L^1(\rrd)},\vsp\hfill 
\ff\lbb>0,\ y_1,y_2\in L^1(\rrd).\earr$$Then, by the Crandall \& Liggett theorem (see \cite{1},  \cite{2}, p. 56) the Cauchy pro\-blem
\begin{equation}\label{e1.6}
\frac{du}{dt}+Au=0,\ t\ge0;\ \ u(0)=u_0,\end{equation}has, for each $u_0\in L^1(\rrd)$ a unique  solution $u=u(t,u_0)$ in the mild sense   \eqref{e1.3}--\eqref{e1.4}. Equivalently,
\begin{equation}\label{e1.7}
u(t,u_0)=\lim_{n\to\9}\(I+\frac tn\,A\)^{-n}u_0\mbox{ in }L^1(\rrd),\end{equation}uniformly on the compact intervals of $[0,\9)$. 

Moreover, the map $t\to u(t,u_0)$, denoted $S(t)u_0$, is a {\it continuous semigroup} of contractions on $L^1(\rrd)$, that is, $$\barr{c}
S(t+s)=S(t)S(s)\mbox{ for all $s,t\ge0$,}\vsp
\|S(t)u_1-S(t)u_2\|_{L^1(\rrd)}\le\|u_1-u_2\|_{L^1(\rrd)}, \ \ff t\ge0,\ u_1,u_2\in L^1(\rrd),\vsp
\dd\lim_{t\to0}S(t)u_0=u_0\mbox{ in }L^1(\rrd).\earr$$Note also (see \cite{3}--\cite{5b}) that \begin{eqnarray}
& S(t)(L^1(\rrd)\cap L^\9(\rrd))\subset L^1(\rrd)\cap L^\9(\rrd),\ \ff t\ge0,\label{e1.7b}\\[1mm]
&S(t)(L^1(\rrd)\cap L^1(\rrd;\Phi dx))\subset L^1(\rrd)\cap L^1(\rrd;\Phi dx),\label{e1.8a}\\[1mm]
&S(t)u_0\in L^\9((0,T)\times\rrd),\ \ff T>0,\ \ff u_0\in L^1(\rrd)\cap L^\9(\rrd),\label{e1.7c}
\end{eqnarray}
and $S(t)\calp\subset\calp,$ $ \ff t\ge0,$ where
\begin{equation}\label{e1.8}
\calp=\left\{y\in L^1(\rrd),\ y(x)\ge0,\mbox{ a.e. } x\in\rrd;\int_\rrd y(x)dx=1\right\}.\end{equation}We also note that, though  $t\to u(t)=S(t)u_0$ is not differentiable,  it is, however,  a Schwartz-distributional solution to \eqref{e1.1}, that~is,

\begin{equation}\label{e1.9}
	\barr{c}
\dd\int^\9_0\int_\rrd(u\vf_t+\beta(u)\Delta_x\vf+b(u)uD\cdot\nabla_x\vf)dx\,dt\\
+\dd  \int_\rrd u_0(x)\vf(0,x)dx=0,\earr\end{equation}for all $\vf\in C^\9_0([0,\9)\times\rrd).$  

Moreover, as shown in \cite{5} (see also \cite{5b}), $S(t)u_0$ is the unique distributional solution to NFPE \eqref{e1.1} in the class of functions $u\in L^1((0,\9)\times\rrd)\cap L^\9((0,\9)\times\rrd)$  such that 
$t\to u(t)dx$ is  weakly continuous on $[0,\9)$.  In~particular, this implies (see, e.g., \cite{5} and \cite{5b}, Chapter~5) that the McKean--Vlasov equation \eqref{e1.2} has a unique strong solution $X_t$ with the marginal law~$u(t,\cdot)$. 

The purpose of this work is to represent the solution $t\to S(t)u_0$ to \eqref{e1.1} as a {\it subgradient flow} of the entropy functional (energy)
\begin{equation}\label{e1.10}
\barr{rl}
E(u)\!\!\!&=\dd\int_\rrd(\eta(u(x))+\Phi(x)u(x))dx,\ u\in\calp\cap L^\9(\rrd)\cap L^1(\rrd;\Phi dx),\vsp 
\eta(r)\!\!\!&=\dd\int^r_0\int^s_1\frac{\beta'(\tau)}{\tau b(\tau)}\,d\tau\,ds,\ r\ge0,\earr \hspace*{-10mm}
\end{equation}
with the tangent space $\calt_u(\calp^*)\subset\cald'(\rrd)$ defined in \eqref{e3.1} below, for  $u\in\calp^*$. Here,  
\begin{equation}\label{e1.10b}
\calp^*=\left\{
\barr{r}
u\in\calp\cap L^\9\cap L^1(\rrd;\Phi dx);\sqrt{u}\in H^1(\rrd),\dd\frac\psi u\in L^1(\rrd)\vsp 
\mbox{ for some }\psi\in\calx\earr\right\},\end{equation} 
where we set $\frac10:=+\9$ and
\begin{equation}\label{e1.13a}	
\hspace*{-5mm}\calx\!=\!\left\{\!\psi\!\in\! C^2(\rrd)\!\cap\! C_b(\rrd)\!\cap\! L^1(\rrd),\psi\!>\!0,\frac{\nabla\psi}{\psi}\!\in\! L^\9(\rrd),\frac1\psi\!\in\! L^1_{\rm loc}(\rrd)\!\right\}\!.\hspace*{-5mm}
\end{equation}
We also note that the function $E$ is convex and lower semicontinuous on $L^2(\rrd)$ with the domain 
$$D(E)=\{u\in\calp\cap L^\9(\rrd)\cap L^\9\cap L^1(\rrd;\Phi dx)\}.$$
The class $\calx$ is clearly nonempty and, in particular, it contains all functions $\psi$ of the form  $\psi(x)=(\alpha_1|x|^m+\alpha_2)\1$, $\alpha_1,\alpha_2$ and $m>d$ and, therefore, since $\calx$ is an algebra containing the constants, it is a rich class of functions. Hence, so is $\calp^*$ since if $\psi\in\calx$, $\psi>0$, $u:=\psi^2\(\int_\rrd\psi^2dx\)\1$ is easily checked to be in $\calp^*$. We also note that $\calp^*$ is convex. The gradient flow representation means that, for $u(t)=S(t)u_0$, $u_0\in\calp^*$, we have
\begin{equation}\label{e1.14}
\frac d{dt}\,u(t)=-\nabla E_{u(t)},\ t>0,
\end{equation}
where $\nabla  E_u\in \calt_u(\calp^*)$ is the  gradient of $E$ in the sense of the Riemannian type geo\-me\-try of $\calp$ to be defined later on. Such a result was recently established in \cite{10} (see also \cite{1prim}, \cite{1secund}, \cite{16})  on the  manifold $\calp$ endowed with the topology of weak convergence of probability measures and tangent bundle $L^2(\rrd;\rrd;\mu)_{\mu\in\calp}$ and in the fundamental work \cite{9} for the classical porous media equation. But we want to emphasize that we consider here the smaller space $\calp^*\subset\calp$ with the tangent bundle $(\calt_u(\calp^*))_{u\in\calp^*}$ defined in \eqref{e3.1} and scalar product \eqref{e3.2} which is different from the one in \cite{1prim}, \cite{1secund}, \cite{9}, \cite{10},  \cite{16}. Herein, we shall obtain a representation of the form \eqref{e1.14} for NFPE \eqref{e1.1}. This result is based on the smoothing effect on initial data of the semigroup $S(t)$ in the space $H\1(\rrd)$ which will be proved in Section \ref{s1}. As a matter of fact, the space $H\1(\rrd)$ is a viable alternative to $L^1(\rrd)$ for proving the well-posedness of NFPE \eqref{e1.1}. In fact, as seen below, the operator \eqref{e1.5} has a quasi-$m$-accretive version in $H\1(\rrd)$, which generates a $C_0$-semigroup of quasi-contractions which coincides with $S(t)$ on $L^1(\rrd)\cap L^\9(\rrd)$. 

We recall that (see, e.g., \cite{1}, \cite{2}), if $H$ is a Hilbert space with the scalar product $(\cdot,\cdot)_H$ and norm $|\cdot|$, the operator $B:D(B)\subset H\to H$ is said to be $m$-accretive if 
$$(Bu_1-Bu_2,u_1-u_2)\ge0,\ \ff  u_i\in D(B),\ i=1,2,$$
and $R(I+\lbb B)=H$, $\ff\lbb>0$. It is said to be quasi $m$-accretive if $B+\oo f$ is $m$-accretive for some $\oo\ge0$.

\bk\noindent{\bf Notation.} $L^p(\rrd),$ $1\le p\le\9$ (denoted $L^p$) is the space of Lebesgue measurable and $p$-integrable functions on $\rr^d$, with the standard norm $|\cdot|_p$. $(\cdot,\cdot)_2$ denotes the inner product in $L^2$. By $L^p_{\rm loc}$ we denote the corresponding local space. Let $C^k(\rr)$ denote the space of continuously differentiable functions up to order $k$ and $C_b(\rr)$ the space of continuous and bounded functions on $\rr$. For any open set $\calo\subset\rr^m$ let $W^{k,p}(\calo)$, $k\ge1$, denote the standard Sobolev space on $\calo$ and by $W^{k,p}_{\rm loc}(\calo)$ the corresponding local space. We set $W^{1,2}(\calo)=H^1(\calo),$ $W^{2,2}(\calo)=H^2(\calo)$, $H^1_0(\calo)=$ $\mbox{$\{u\in H^1(\calo),\,u=0$}\mbox{ on }\pp\calo\}$, where $\pp\calo$ is the boundary of $\calo$. By $H\1(\calo)$ we denote the dual space of $H^1_0(\calo)$ (of~$H^1(\rr^m)$, respectively, if $\calo=\rr^m$). We shall also set $H^1=H^1(\rrd)$ and $H\1=H\1(\rrd)$. 
 $C^\9_0(\calo)$ is the space of infinitely differentiable real-valued functions with compact support in $\calo$ and $\cald'(\calo)$ is the dual of $C^\9_0(\calo)$, that is, the space of Schwartz distributions on $\calo$. ${\rm Lip}(\rr)$ is the space of  real-valued Lipschitz functions on $\rr$ with the norm denoted by $|\cdot|_{\rm Lip}$. The space $H\1$ is endowed with the scalar product  
 $$\<y_1,y_2\>_{-1}=((I-\Delta)\1y_1,y_1)_2,\ \ff y_1,y_2\in H\1,$$
 and the Hilbert norm $|y|^2_{-1}=\<y,y\>_{-1}.$ 
By ${}_{H\1}(\cdot,\cdot)_{H^1}$ we denote the dua\-li\-ty pairing on $H^1\times H\1.$ If $Y$ is a Banach space,  
  then $C([0,\9);Y)$ is the space of continuous functions  $y:[0,\9)\to Y$ and $C_w([0,\9);Y)$ is the space of weakly continuous $Y$-valued functions. Furthermore, let  $C^\9_0([0,\9)\times\rrd)$ denote the space of all $\vf\in C^\9([0,\9)\times\rrd)$ such that $support\,\vf\subset K$, where $K$ is compact in $[0,\9)\times\rrd.$ If $u:[0,\9)\to H\1$ is a given function, we shall denote its $H\1$-strong derivative in $t$ by $\frac{du}{dt}\,(t)$, and  the right derivative by $\frac{d^+}{dt}\,u(t)$. 
    We  shall also use the following  notations     
 $$\barr{c}
\beta'(r)\equiv\dd\frac d{dr}\,\beta(r),\ b'(r)=\frac d{dr}\,b(r),\ b^*(r)\equiv b(r)r,\ r\in\rr,\\ 
y_t=\dd\frac\pp{\pp t}\,y,\ \nabla y=\left\{\frac{\pp y}{\pp x_i}\right\}^d_{i=1},\ \Delta y=\sum^d_{i=1}\frac{\pp^2}{\pp^2x_i}\,y,\earr$$ 
$$\divv\,y=\dd\sum^d_{i=1}\frac{\pp y_i}{\pp x_i},\ y=\{y_i\}^d_{i=1},$$for $y=y(t,x),\ (t,x)\in[0,\9)\times\rrd$, where $\Delta$ and $\divv$ are taken in the sense of the distribution space $\cald'(\rrd)$.

\section{The $H^1$-regularity of the semigroup $S(t)$}\label{s2}
\setcounter{equation}{0}

Consider the semigroup $S(t):L^1\to L^1$ defined earlier by the exponential formula \eqref{e1.7}. 
Define the operator $A^*:H\1\to H\1,$ 
\begin{equation}\label{e2.1}
A^*y=-\Delta\beta(y)+\divv(Db^*(y)),\ \ff y\in D(A^*),\end{equation}with the domain $D(A^*)=H^1.$  More precisely, for each $y\in H^1$,  $A^*y\in H\1$ is defined by
\begin{equation}
\label{e2.1a}
{}_{H\1}(A^*y,\vf)_{H^1}=\int_\rrd(\nabla\beta(y)-Db^*(y))\cdot\nabla\vf\,dx,\ \ff\vf\in H^1.
\end{equation}
As mentioned earlier, the semigroup $S(t)$ is not differentiable in $L^1$, but as shown below it is, however, $H\1$-differentiable on the right on $(0,\9)$. 

Namely, we have

\begin{theorem}\label{t1} Assume that Hypotheses {\rm(i)--(iv)} hold. Then, for each $u_0\in\calp\cap L^\9$, the function $u(t)=S(t)u_0$ is in $C([0,\9);H^{-1})\cap C_w([0,\9);L^2)$, it is $H\1$-right differentiable on $(0,\9)$ with $\frac{d^+}{dt}\,u(t)$ being  $H\1$-continuous from the right on $(0,\9)$, $S(t)u_0\in H^1$, $t>0$,  and	
\begin{equation}\label{e2.2}
\frac{d^+}{dt}\,S(t)u_0+A^*S(t)u_0=0,\ \ff t>0.\end{equation}
Furthermore, $S(t)u_0\in\calp\cap L^\9$, $\ff t\ge0$, $\frac{d}{dt}\,S(t)u_0$ exists on $(0,\9)\setminus N$, where $N$ is an at most countable subset of $(0,\9)$, 
\begin{equation}\label{e2.3}
\frac d{dt}\,S(t)u_0+A^*S(t)u_0=0,\ \ff t\in(0,\9)\setminus N,
\end{equation}and  $t\to A^*S(t)u_0$ is $H\1$-continuous on $(0,\9)\setminus N$.

 Moreover, $\sqrt{S(t)u_0}\in H^{1}(\rrd)$, a.e. $ t>0$, that is, 
 \begin{eqnarray}
&\dd\dd\frac{\nabla S(t)u_0}{\sqrt{S(t)u_0}}\in L^2,\ \mbox{a.e. } t\in(0,\9),\label{e2.4}\end{eqnarray}
 \begin{eqnarray}
 	& E(S(t)u_0)<\9, \mbox{ a.e. } t\in(0,\9).\label{e2.4a}\end{eqnarray}
for all $u_0\in\calp$ such that $u_0\log u_0\in L^1$. If $u_0\in H^1$, then \eqref{e2.2} holds for all $t\ge0$,  
$t\to S(t)u_0$ is 
locally $H^{-1}$-Lipschitz, on $[0,\9)$ and $u(t)\in H^1$, $\ff t\ge0$.

Finally, if $u_0\in \calp^*$, then
\begin{equation}\label{e2.5a}
S(t)u_0\in\calp^*,\ \ff t\ge 0.
\end{equation}
\end{theorem}
In particular, it follows by Theorem \ref{t1} that the semigroup $S(t)$ is ge\-ne\-rated by the opertor $-A^*$ in the space $H\1$.  

We shall prove Theorem \ref{t1} in several steps, the first one being the following lemma.
	
\begin{lemma}\label{l2} The operator $A^*$ is quasi-$m$-accretive in $H\1$, that is, $A^*+\oo I$ is $m$-accretive for some $\oo\ge0$.
\end{lemma}	

\begin{proof}
	We have
\begin{eqnarray}
&&\hspace*{-12mm}\<A^*y_1-A^*y_2,y_1-y_2\>_{-1}\label{e2.5}\\[1mm]
&&=(\beta(y_1){-}\beta(y_2),y_1-y_2)_2-
(\beta(y_1){-}\beta(y_2),(I{-}\Delta)\1(y_1{-}y_2))_2\nonumber\\[1mm]
&&\qquad\qquad\qquad\qquad\qquad+(D(b^*(y_1){-}b^*(y_2)),\nabla
(I{-}\Delta)\1(y_1{-}y_2))_2\nonumber\\[1mm]
&&\ge\gamma_1|y_1{-}y_2|^2_2{-}\gamma_2|y_1{-}y_2|_{-1}|y_1{-}y_2|_2-|D|_\9|b^*|_{\rm Lip}
|y_1{-}y_2|_2|y_1{-}y_2|_{-1}\nonumber\\[1mm]
&&\ge-\oo|y_1{-}y_2|^2_{-1},\ \ff y_1,y_2\in D(A^*),\nonumber
\end{eqnarray}for a suitable chosen $\oo\ge0$ and so $A^*+\oo I$ is accretive in $H\1$. (Here, we have used the inequality $|\nabla(I-\Delta)\1(y_1-y_2)|_2\le|y_1-y_2|_{-1}$.) Now, we shall prove that $R(I+\lbb A^*)=H\1$ for  $\lbb\in(0,\lbb_0)$, where $\lbb_0$ is suitably chosen. For this purpose, we fix $f\in H\1$ and consider the equation  
 \begin{equation}\label{e2.6}
y-\lbb\Delta\beta(y)+\lbb\divv(Db^*(y))=f\mbox{ in }\cald'(\rrd),\ y\in L^2.\end{equation}The latter can be written as
\begin{equation}\label{e2.7a}
G_\lbb(y)=(I-\Delta)\1f,\end{equation}where $G_\lbb:L^2\to L^2$ is the  operator
$$\barr{r}
G_\lbb(y)=\lbb\beta(y)+(I-\Delta)\1y-\lbb(I-\Delta)\1\divv(Db^*(y))-\lbb(I-\Delta)\1\beta(y),\\ \ff y\in L^2\earr,$$
which by Hypotheses (i)--(iii) is continuous. 
Then, by (i)--(iii) we have

$$\barr{l}
(G_\lbb(y_1)-G_\lbb(y_2),y_1-y_2)_2\vsp
\hspace*{16mm}\ge\lbb\gamma_1|y_1{-}y_2|^2_2
{+}|y_1{-}y_2|^2_{-1}{-}\lbb \gamma_2|y_1{-}y_2|_{-1}|y_1{-}y_2|_2\vsp
\hfill{-}\lbb|D|_\9|b^*|_{\rm Lip}|y_1{-}y_2|_2|y_1{-}y_2
|_{-1}\vsp
\hspace*{16mm}\ge\dd\frac12(\lbb\gamma_1{-}\gamma^2_2\gamma\lbb^2-\lbb^2|D|^2_\9|b^*|^2_{\rm Lip})|y_1{-}y_2|^2_2
{+}\dd\frac12|y_1{-}y_2|^2_{-1}\vsp 
\hspace*{16mm}\ge\alpha|y_1{-}y_2|^2_2,\ \ff y_1,y_2\in L^2,\earr$$
for some $\alpha>0$ and $0<\lbb<\lbb_0$ with $\lbb_0$ sufficiently small. Hence, the operator $G_\lbb$ is monotone and coercive in the space $L^2$. Since 
it is also continuous, we infer that it is surjective (see, e.g., \cite{1}, p.~37) and, therefore, $R(G_\lbb)=L^2$ for $0<\lbb<\lbb_0)$. Hence, \eqref{e2.7a} (equivalently \eqref{e2.6}) has a solution $y\in L^2$ for $\lbb\in(0,\lbb_0)$ and $\beta(y)\in H^1$. Then, by (i) it follows that $y\in H^1$ and so $y\in D(A^*)$. Hence, $A^*$ is quasi-$m$-accretive in $H\1$.\end{proof}

 Lemma \ref{l2} implies that there is a  $C_0$-continuous nonlinear semigroup $S^*(t):H\1\to H\1$, $t\ge0$, which is generated by $-A^*$. This means (see, e.g., \cite{1}, p.~146 or  \cite{2}, p.~56) that
\begin{equation}\label{e2.7}
S^*(t)u_0=\lim_{n\to\9}\(I+\frac tn\,A^*\)^{-n}u_0\ \mbox{ in }H\1,\ \ff t\ge0,\ \ff u_0\in H\1,\end{equation}uniformly on compact intervals. Moreover, for all $u_0\in D(A^*)=H^1$ we have $S^*(t)u_0\in D(A^*),$ $[0,\9)\ni t\mapsto S^*(t)u_0\in H\1$ is locally Lipschitz and,
\begin{eqnarray}
\frac{d^+}{dt}\,S^*(t)u_0+A^*S^*(t)u_0&=&0,\ \ \ff t\ge0,\label{e2.8}\\[1mm]
\frac{d}{dt}\,S^*(t)u_0+A^*S^*(t)u_0&=&0,\ \mbox{ a.e. } t>0,\label{e2.10a}
 \end{eqnarray}
 and the function $t\to\frac{d^+}{dt}\,S^*(t)u_0$ is continuous from the right in the  $H\1$-topology.  
 Taking into account \eqref{e2.1a}, we can rewrite \eqref{e2.10a} as
 \begin{equation}\label{e2.12a}
 \barr{l}
 \dd\frac d{dt}\int_\rrd(S^*(t)u_0)(x)\vf(x)dx
 +\dd\int_\rrd(\nabla\beta(S^*(t)u_0)(x))\vsp\qquad-D(x)b^*((S^*(t)u_0)(x))\cdot\nabla\vf(x)dx=0,\ \mbox{a.e. }t>0,\ \ff\vf\in H^1.\earr
 \end{equation} 
We also note that the semigroup $S^*(t)$ is quasi-contractive on $H\1$, that is,
$$|S^*(t)u_0-S^*(t)\bar u_0|_{-1}\le\exp(\oo t)|u_0-\bar u_0|_{-1},\ \ff t\ge0,\ \ff u_0,\bar u_0\in H\1,$$for some $\oo\ge0$. 
 Moreover, we have for all $u_0\in L^2$ and $T>0$,
\begin{equation}\label{e2.9}
|S^*(t)u_0|^2_2+\int^t_0|\nabla(S^*(s)u_0)|^2_2ds\le C_T|u_0|^2_2,\ \ff t\in[0,T].\end{equation}Here is the argument. By \eqref{e2.7} we have, for all $T>0$,
\begin{equation}\label{e2.10}
S^*(t)u_0=\lim_{h\to0} v_h(t)\ \mbox{ in }H\1,\ \ff t\in(0,T),\end{equation}where

\begin{equation}\label{e2.11}
\barr{rcl}
v_h(t)&\!\!=\!\!&v^j_h,\ \ff t\in[jh,(j+1)h),\ j=0,1,...,N_h=\mbox{$\left[\frac Th\right]$},\vsp 
v^{j+1}_h+hA^*v^{j+1}_h&\!\!=\!\!&v^j_h,\ j=0,1,...,N_h;\ \ v^0_h=u_0.\earr
\end{equation}\newpage \n Since $v^0_h=u_0\in L^2,$ we get by \eqref{e2.11}
$$\barr{c}
(\beta(v^{j+1}_h),v^{j+1}_h-v^j_h)_2+h|\nabla\beta(v^{j+1}_h)|^2_2
= h(\nabla\beta(v^{j+1}_h),Db^*(v^{j+1}_h))_2\vsp 
\le\dd\frac h2\,|\nabla\beta(v ^{j+1}_h)|^2_2+\frac h2\,(|D|^\9|b^*|_\9|v^{j+1}_h|_2)^2.\earr$$
By (i), this yields
$$\int_\rrd j(v^{j+1}_h)dx+\frac12\,\gamma^2_1h\sum^{j+1}_{k=1}|\nabla(v^{k}_h)|^2_2\le\int_\rrd j(u_0)dx+Ch\dd\sum^{j+1}_{k=1}|v^k_h|^2_2,$$
where $j(r)=\int^r_0\beta(s)ds.$ Since $\frac12\,\gamma_1 r^2\le j(r)\le\frac12\,\gamma_2r^2,$ $\ff r\in\rr$, we have 
$$|v_h(t)|^2_2+\int^t_0|\nabla v_h(s)|^2_2ds\le C\(\int^t_0|v_h(s)|^2_2ds+|u_0|^2_2\),\ t\in(0,T).$$
Hence
$$|v_h(t)|^2_2+\int^t_0|\nabla v_h(s)|^2_2ds\le C|u_0|^2_2,\ \ff t\in(0,T),\ h>0.$$
Therefore, by \eqref{e2.10} and by the weak-lower semicontinuity of the $L^2(0,T;H^1)$-norm,  \eqref{e2.9} follows. Hence,   $S^*(t)u_0\in H^1$, a.e. $t>0$, and so, by the semigroup property, $S^*(t+s)=S^*(t)S^*(s),$ $t,s\ge0,$ we infer that  $S^*(t)$ has a smoothing effect on initial data, that is, 
\begin{equation}\label{e2.12}
S^*(t)u_0\in H^1=D(A^*),\ \ff t>0, u_0\in L^2.\end{equation}
Then, by \eqref{e2.8} it follows that $t\mapsto S^*(t)u_0$ is $H\1$-continuous on $(0,T)$ for all $u_0\in L^2$, hence $t\mapsto|S^*(t)u_0|_2$ is lower semicontinuous on $(0,T)$. Furthermore, \eqref{e2.12} implies
\begin{equation}\label{e2.13}
\frac{d^+}{dt}\,S^*(t)u_0+A^*S^*(t)u_0=0,\ \ff u_0\in L^2,\ \ff t>0,\end{equation}
and that the function $t\to \frac{d^+}{dt}\,S^*(t)u_0=-A^*S^*(t)u_0$ is $H\1$-right continuous on~$(0,\9)$. Since $S^*(\cdot)u_0\in L^\9(0,T;L^2)\cap C([0,T];H\1)$, it follows that $$\sup_{t\in[0,T]}|S^*(t)u_0|_2\le{\rm ess}\sup_{\hspace*{-4mm}t>0}|S^*(t)u_0|_2\vee|S^*(T)u_0|_2+|u_0|_2<\9$$ and hence we obtain
by the uniqueness of limits that the function $t\to S^*(t)u_0$ is $L^2$-weakly continuous, that is, $S^*(\cdot)u_0\in C_w([0,T];L^2)$,\ $\ff T>0$.  
We set $u_h(t)=u(t+h)-u(t),$ $u(t)\equiv S^*(t)u_0$, $\ff t\in[0,T],$ $h>0$, $u_0\in L^2$. By   \eqref{e2.13} we have\newpage 
$$\frac{d^+}{dt}\,u_h(t)+A^*u(t+h)-A^*u(t)=0,\ \ff t\in(0,T].$$
This yields
$$\frac12\,\frac{d^+}{dt}\,|u_h(t)|^2_{-1}+\<A^*u(t+h)-A^*u(t),u_h(t)\>_{-1}=0$$and, therefore, by  \eqref{e2.5} $$\frac12\,\frac{d^+}{dt}\,|u_h(t)|^2_{-1}\le\oo|u_h(t)|^2_{-1},\ \ff t\in(0,T].$$
Hence,  for all $h>0$, we have
$$|u_h(t)|_{-1}\exp(-\oo t)\le|u_h(s)|_{-1}\exp(-\oo s), \ 0<s<t<T,$$and, therefore, the function  $t\to\exp(-\oo t)
|A^*S^*(t)u_0|_{-1}$ is monotonically de\-crea\-sing on $(0,\9)$ and so it is everywhere continuous on  $(0,\9)$, except for a countable set $N\subset(0,\9)$.  

Since the continuity points of $t\to \exp(-\oo t)A^*S^*(t)u_0$   coin\-cide with that of  $t\to\exp(-\oo t)|A^*S^*(t)u_0|_{-1}$ (see the proof of Lemma 3.1 in \cite{6}), we infer that the function $t\to\exp(-\oo t) A^*S^*(t)u_0$ has at most countably many discontinuities. Hence, for each $u_0\in L^2$, the function $t\to S^*(t)u_0$ is $H\1$ differentiable on $(0,\9)\setminus N$~and
\begin{equation}\label{e2.14}
\frac d{dt}\,S^*(t)u_0+A^*S^*(t)u_0=0,\ \ff t\in(0,\9)\setminus N,\end{equation}where $N$ is a countable subset of $(0,\9)$.\bk

\n{\it Proof of Theorem} \ref{t1}  (continued).  We note first that
\begin{equation}\label{e2.15}
S(t)u_0=S^*(t)u_0,\ \ff t\ge0,\ u_0\in L^1\cap L^2.\end{equation}
Indeed, by \eqref{e2.11}  it follows that if $u_0\in L^1\cap L^2$, then $|v^{j+1}_h|_1\le|v^j_k|_1,$\break $ \ff j=0,1,...$, and, therefore,
\begin{equation}\label{e2.21b}
|v^{j+1}_h|_1\le|v^0_j|_1=|v_0|_1,\ \ff j.\end{equation}
This follows by multiplying equation \eqref{e2.11}  with $\calx_\delta(v^{j+1}_h)$ and integrating over  $\rrd$, where $\calx_\delta$ is defined by\newpage
$$\calx_\delta(r)=\left\{\barr{rll}
1&\mbox{ for }&r\ge\delta,\vsp 
\dd\frac r\delta&\mbox{ for }&r\in(-\delta,\delta),\vsp 
-1&\mbox{ for }&r\le-\delta.\earr\right.$$
Taking into account that $v^{j+1}_h\in H^1$, we have by \eqref{e2.1} that
$$\barr{ll}
(A^*v^{j+1}_h,\calx_\delta(v^{j+1}_h))_2\!\!\!\vsp\dd=\int_\rrd\beta'(v^{j+1}_h)|\nabla v^{j+1}_h|^2\calx'_\delta(v^{j+1}_h)dx 
+\dd\int_{[x;|v^{j+1}_h(x)|\le\delta]}(v^{j+1}_h)(D\cdot\nabla v^{j+1}_h)dx,\earr$$
which yields
$$\limsup_{\delta\to0}(A^*v^{j+1}_h,\calx_\delta(v^{j+1}_h))_2\ge0,\ \ff j=0,1,...\,.$$
Hence, 
$$\limsup_{\delta\to\9}\int_\rrd v^{j+1}_h\calx_\delta(v^{j+1}_h)dx\le|v^j_h|_1,\ \ff j=0,1,...,$$and so \eqref{e2.21b} follows. 

Comparing \eqref{e2.11}  with \eqref{e1.4}, we infer that $u_h\equiv v_h$, $\ff h$, and so, by \eqref{e1.3} and \eqref{e2.10}, we get \eqref{e2.15}, as claimed. In particular, we have that, if $u_0\in\calp\cap L^\9$, then by \eqref{e1.7b} it  also follows that $S^*(t)u_0\in\calp\cap L^\9$, $\ff t>0$.

Roughly speaking, this means that the semigroup $S(t)$ is smooth on $L^1\cap L^2$ in $H\1$-norm. Then, by  \eqref{e2.2}--\eqref{e2.3}, \eqref{e2.15} and the corresponding properties of $S(t)$ follow by \eqref{e2.8}, \eqref{e2.13}--\eqref{e2.14}. As regards \eqref{e2.4}--\eqref{e2.4a}, we note first that by Theorem 4.1 in \cite{4} (see also \cite{5b}, p.~161), we have, for all $u_0\in\calp$ with $u_0\log u_0\in L^1$,
\begin{equation}\label{e2.17}
E(S(t)u_0)+\int^t_0\Psi(S(\tau)u_0)d\tau\le E(u_0)<\9,\ \ff t\ge0,\end{equation}where $E$ is the energy functional  \eqref{e1.10} and  
\begin{equation}\label{e2.23a}
\Psi(u)\equiv\int_\rrd
\left|\frac{\beta'(u)\nabla u} {\sqrt{b^*(u)}}-D\sqrt{b^*(u)}\right|^2dx.\end{equation}
Hence, $\Psi(S^*(t)u_0)<\9,$ a.e. $t>0$, which by \eqref{e2.15} and Hypotheses (i)--(iii) implies \eqref{e2.4} (see \cite[Lemma 5.1]{4}), as claimed. Moreover, by \eqref{e2.17},    also \eqref{e2.4a} holds.
 
 Assume now that {$u_0\in \calp^*$, hence $\frac\psi{u_0}\in L^1$ for some $\psi\in\calx$, where $\calx$ is defined by \eqref{e1.13a}. We note that since $S(t)(\calp)\subset\calp$, $\ff t\ge0$, we have also that $u(t)\ge0$, $\ff t\ge0$, and $u(t)\in L^\9$, $\ff t\ge0$. So, it remains to prove that $\frac\psi{u(t)}\in L^1$,  $\ff t\ge0$.} To this end, we consider the cut-off function $$\vf_n(x)=\eta\(\frac{|x|^2}n\)\psi(x),\ \ff x\in\rrd,\ n\in\nn,$$
 where $\eta\in C^2([0,\9))$ is such that $0\le\eta\le1$ and
 \begin{equation}
 \label{e2.19a}
 \eta(r)=1,\ \ff r\in[0,1];\ \ \eta(r)=0,\ \ff r\ge2.\end{equation}  
Since $u:[0,\9)\to H\1$ is locally Lipschitz,    $[0,\9)\ni t\to{}_{H\1}(u(t),\vf)_{H^1}$ is locally Lipschitz for all $\vf\in H^1$, and so almost everywhere differentiable.  We also note the chain differentiation   rule 
$$\frac d{dt}\int_\rrd g(u(t,x))\vf_n(x)dx={}_{H\1}\!\(\frac{du}{dt}\,(t),\gamma(u(t))\vf_n\)_{H^1},\mbox{ a.e. }t\in(0,T),$$
for all $T>0$ and all $u\in L^2(0,T;H^1),$ with $\frac{du}{dt}\in L^2(0,T;H\1)$, where $\gamma\in C^1(\rr)$, $g(r)\equiv\int^r_0\gamma(s)ds$. 

 In the special case, where $\frac{du}{dt}\in L^2(0,T;L^2)$, this formula follows by \cite[Lemma 4.4, p.~158]{1}. If $\frac{du}{dt}\in L^2(0,T;H\1)$, this follows by approximating $u$ by $u_\vp=(I-\vp\Delta)\1u$ and letting $\vp\to0$. We also note that by \eqref{e2.9} we have that $u=S^*(t)u_0\in L^2(0,T;H^1)$. 

 Let $\vp>0$ be arbitrary, but fixed. Then, since $(u(\cdot)+\vp)\1\in L^2(0,T;H^1)$, we~have
 $$-\frac d{dt}\int_\rrd\frac{\vf_n(x)}{u(t,x)+\vp}\,dx={}_{H\1}\(\frac{du}{dt}\,(t),\frac{\vf_n}{(u(t)+\vp)^2}\)_{H^1},\ \mbox{a.e. }t>0,$$and so, by \eqref{e2.12a} we get
 
 \begin{equation}
 \label{e2.19aa}
 \quad\barr{l}
 \dd\frac d{dt}\int_\rrd\frac{\vf_n(x)}{u(t,x)+\vp}\,dx
 +2\int_\rrd\frac{\beta'(u(t,x))\vf_n(x)|\nabla u(t,x)|^2}{(u(t,x)+\vp)^3}\,dx\vspace*{3mm}\\ 
 \quad\quad=\dd\int_\rrd\frac{\beta'(u(t,x))(\nabla\vf_n(x)\cdot\nabla u(t,x))}{(u(t,x)+\vp)^2}\,dx\vspace*{3mm}\\ 
 \quad\qquad-\dd\int_\rrd\frac{(D(x)\cdot\nabla\vf_n(x))b(u(t,x))u(t,x)}{(u(t,x)+\vp)^2}\,dx\vspace*{3mm}\\ 
 \quad\qquad+2\dd\int_\rrd\frac{b(u(t,x))u(t,x))(D(x)\cdot\nabla u(t,x))\vf_n(x)}{(u(t,x)+\vp)^3}\,dx,
 \mbox{ a.e. }t>0.\earr\hspace*{-10mm}
 \end{equation}
 By \eqref{e2.19a}, we have
 $$|\nabla\vf_n(x)|\le\frac{4\psi(x)}{\sqrt{n}}\,|\eta'|_\9+\vf_n(x)
g(x),\ x\in\rrd,$$
 {where $g(x)=\frac{|\nabla\psi(x)|}{\psi(x)}$.}\mk  
 
 On the other hand, we have by Hypotheses (i)---(iii) that
  \begin{equation}\label{2.26} 
 \barr{l}
 2\dd\int_\rrd\frac{\beta'(u(t,x))\vf_n(x)|\nabla u(t,x)|^2}{(u(t,x)+\vp)^3}\,dx
 \ge2\gamma_1\dd\int_\rrd\frac{\vf_n(x)|\nabla u(t,x)|^2}{(u(t,x)+\vp)^3},\earr
 \end{equation}
 and
 
 \begin{equation}\label{2.27} 
 \barr{l}
 \dd\int_\rrd
 \frac{\beta'(u(t,x))\nabla\vf_n(x)\cdot\nabla u(t,x)}{(u(t,x)+\vp)^2}\,dx\vspace*{3mm}\\
 \qquad
 \le\gamma_2\dd\int_\rrd\frac{|\nabla u(t,x)|}{(u(t,x)+\vp)^2}\,
 \(\dd\frac{4\psi(x)}{\sqrt{n}}\,|\eta'|_\9+\vf_n(x)
 g(x)\)dx\vspace*{3mm}\\
 \qquad 
 \le C_1\gamma_2\dd\int_\rrd\frac{|\nabla u(t,x)|\vf_n(x)}{(u(t,x)+\vp)^2}\,dx
 +\dd\frac{C_2\gamma_2}{\sqrt{n}}\dd\int_\rrd
 \frac{|\nabla u(t,x)|\psi(x)}{(u(t,x)+\vp)^2}\,dx\vspace*{3mm}\\
 \qquad\le\dd\frac{\gamma_1}2\dd\int_\rrd\frac{\vf_n(x)|\nabla u(t,x)|^2}{(u(t,x)+\vp)^3}\,dx
 +\dd\frac{C_2\gamma_2}{\sqrt{n}}
 \int_\rrd\frac{|\nabla u(t,x)|\psi(x)}{(u(t,x)+\vp)^2}\,dx\vspace*{3mm}\\
 \qquad+\,C_3\dd\int_\rrd\frac{\vf_n(x)}{u(t,x)+\vp}\,dx.
  \earr\end{equation}
 
 \begin{equation}\label{2.28} 
 \barr{l}
 \dd\int_\rrd
 \frac{D(x){\cdot}\nabla\vf_n(x)b(u(t,x))u(t,x)}{(u(t,x)+\vp)^2}\,dx
 \le C_4\dd\int_\rrd\frac{|\nabla\vf_n(x)|}{u(t,x)+\vp}\,dx\vspace*{3mm}\\
 \qquad
 \le C_5\dd\int_\rrd\(\frac{\vf_n(x)}{u(t,x)+\vp}+\frac1{\sqrt{n}(u(t,x)+\vp)}\)dx.\earr\end{equation}
 
 \begin{equation}\label{2.29} 
 \barr{l}
 2\dd\int_\rrd
 \frac{b(u(t,x))u(t,x)(D(x){\cdot}\nabla u(t,x))\vf_n(x)}{(u(t,x)+\vp)^3}\,dx\vspace*{3mm}\\
 \qquad
 \le C_6\gamma_3\dd\int_\rrd\frac{|\nabla u(t,x)|\vf_n(x)}{(u(t,x)+\vp)^2}\,dx\vspace*{3mm}\\
 \qquad
 \le\dd\frac{\gamma_1}2\int_\rrd\frac{|\nabla u(t,x)|^2\vf_n(x)}{(u(t,x)+\vp)^3}\,dx
 +C_7\dd\int_\rrd\frac{\vf_n(x)}{u(t,x)+\vp}\,dx.
 \earr\end{equation}
  Then,  by \eqref{2.26}--\eqref{2.29} and by Hypotheses (i)--(iii) it follows that 
 $$\barr{ll}
 \dd\frac d{dt}\int_\rrd\frac{\vf_n(x)}{u(t,x)+\vp}\,dx\!\!\!
 &+\gamma_1\dd\int_\rrd
 \frac{\vf_n(x)|\nabla u(t,x)|^2}{(u(t,x)+\vp)^3}\,dx\vspace*{3mm}\\
 &\le C_8\dd\int_\rrd\frac{\vf_n(x)}{u(t,x)+\vp}\,dx+\dd\frac{C_8}{\sqrt{n}}\int_\rrd\frac{1}{u(t,x)+\vp}\,dx\vspace*{3mm}\\
 &\ \ \ +\,\frac{C_8}{\sqrt{n}}
 \dd\int_\rrd\frac{|\nabla u(t,x)|\psi(x)}{(u(t,x)+\vp)^2}\,dx,\mbox{ a.e. }t>0.\earr$$
 This yields
 $$\barr{ll}
 \dd\int_\rrd\frac{\vf_n(x)}{u(t,x)+\vp}\,dx\!\!\!
 &+\gamma_1\dd\int^t_0\int_\rrd
 \frac{\vf_n(x)|\nabla u(s,x)|^2}{(u(s,x)+\vp)^3}\,dxds\vspace*{3mm}\\
 &\le\dd\int_\rrd\frac{\vf_n(x)}{u_0(x)+\vp}\,dx
 +C_9\dd\int^t_0ds\int_\rrd\frac{\vf_n(x)}{u(s,x)+\vp}\,dx\vspace*{3mm}\\
 &\ \ +\,
 \dd\frac{C_9}{\sqrt{n}} \int^t_0ds\int_\rrd\frac{1}{u(s,x)+\vp}\,dx\vspace*{3mm}\\
 &+\,\dd\frac{C_9}{\sqrt{n}}\int^t_0 ds\int_\rrd
 \frac{|\nabla u(s,x)|\psi(x)}{(u(s,x)+\vp)^2}\,dxds,\  \ff t\ge0,\earr$$
 while 
 $$\barr{ll}
 \dd\int^T_0\!\!\int_\rrd\!\frac{|\nabla u(s,x)|\psi(x)}{(u(s,x)+\vp)^2}\,dxds
\!\!\!& \le\!\dd\frac1{\vp^{2}}
 \(\dd\int^T_0\!\!ds\int_\rrd|\nabla u(s,x)|^2
 dx\!\)^{\!\!\!\frac12}\!\!
 \(T\!\!\int_\rrd\!\psi^2(x)dx\)^{\frac12}\vspace*{3mm}\\
 &\le\dd\frac{C_{10}}{\vp^{2}},
 \earr$$ 
 because by \eqref{e2.9} we know that $\nabla u\in L^2(0,T;L^2)$.
 
 Letting $n\to\9$, we get
 $$\int_\rrd\frac{\psi(x)dx}{u(t,x)+\vp}\le\int_\rrd\frac{\psi(x)dx}{u_0(x)+\vp}+C_T\int^t_0ds\int_\rrd\frac{\psi(x)dx}{u(s,x)+\vp},\ \ff t\in(0,T),$$ where $C_T>0$ is independent of $\vp$, and so, for $\vp\to0$ it follows by Gronwall's lemma (which is applicable since $\psi\in L^1$) and by Fatou's lemma,
 $$\int_\rrd\frac{\psi(x)dx}{u(t,x)}\le\exp(C_Tt)\int_\rrd\frac{\psi(x)dx}{u_0(x)}<\9,\ \ff t\in(0,T),$$as claimed.\hfill$\Box$

\section{A new tangent space to $\calp$}
\label{s3}
\setcounter{equation}{0}

To represent NFPE \eqref{e1.1} as a gradient flow as in \cite{9}, \cite{10}, we shall  interpret  the space $\calp^*$ as a Riemannian manifold endowed with an appropriate tangent bundle with scalar product which is, however, different from the one in \cite{9}, \cite{10}. To this purpose, we define the tangent space $\calt_u(\calp^*)$  at $u\in\calp^*\subset\calp$ as follows,   
\begin{equation}
\label{e3.1}
\calt_u(\calp^*)=\{z=-\divv(b^*(u)\nabla y); y\in W^{1,1}_{\rm loc}(\rrd),  \sqrt{u}\,\nabla y\in L^2\}\,(\subset H\1). \end{equation}
(Here, $\calp^*$ is defined by \eqref{e1.10b}.) 

The differential structure of the manifold $\calp^*$ is defined by providing for $u\in\calp^*$ the linear space $\calt_u(\calp^*)$ with the scalar product (metric tensor)
\begin{equation}
	\barr{rcl}
	\<z_1,z_2\>_u&=&\dd\int_\rrd b^*(u)\nabla y_1\cdot
\nabla y_2\,dx,\vsp 
 z_i&=&{\rm div}(b^*(u)\nabla y_i),\ i=1,2,\earr\label{e3.2}\end{equation}
and with the corresponding Hilbertian norm  $\|\cdot\|_u$,
\begin{equation}
\|z\|^2_u=\dd\int_\rrd b^*(u)|\nabla y|^2dx,\ z=-{\rm div}(b^*(u)\nabla y).\label{e3.2a}
\end{equation}
As a matter of fact, $\calt_u(\calp^*)$ is viewed here as a factor space by identifying in \eqref{e3.2} two functions $y_1,y_2\in W^{1,1}_{\rm loc}$ if ${\rm div}(b^*(u)\nabla(y_1-y_2))\equiv0$. Note also that, since $b^*(u)\ge b_0u$ and $u>0$, a.e. on $\rrd$,  $\|z\|_u=0$ implies that $z\equiv0$. Moreover, we have
\begin{equation}
\label{e3.3}
\|z_1\|_u=\|z_2\|_u\mbox{ for }z_i={\rm div}(b^*(u)\nabla y_i),\ i=1,2,
\end{equation}
if ${\rm div}(b^*(u)\nabla(y_1-y_2))\equiv0$ in $H\1$. Indeed, for each $\vf\in C^\9_0(\rrd)$, we have in this~case that
$$\int_\rrd b^*(u)\nabla(y_1-y_2)\cdot\nabla(\vf y_i)dx=0,\ i=1,2,$$and this yields
\begin{equation}\label{e3.4}
\int_\rrd b^*(u)\nabla(y_1-y_2)\cdot(\vf\nabla y_i+y_i\nabla\vf)dx=0,\ i=1,2.
\end{equation}
If we take $\vf(x)=\eta\(\frac{|x|^2}n\),$ where $\eta\in C^2([0,\9));\ \eta(r)=1$ for $0\le r\le1,$ $\eta(r)=0$ for $r\ge2$, and let $n\to\9$ in \eqref{e3.4}, we get via the Lebesgue dominated convergence theorem that\newpage
$$\int_\rrd b^*(u)\nabla(y_1-y_2)\cdot\nabla y_i\,dx=0,\ i=1,2,$$
which, as easily seen, implies \eqref{e3.3}, as claimed. Hence, the norm $\|z\|_u$ is independent of representation \eqref{e3.2} for $z$. 
We should also note that $\calt_u(\calp^*)$ so defined is a Hilbert space, in particular, it is complete in the norm $\|\cdot\|_u$. Here is the argument.

Let $u\in\calp^*$ and let $\{y_n\}\subset W^{1,1}_{\rm loc}$ be such that
$$\|z_n-z_m\|^2_u=\int_\rrd b^*(u)|\nabla(y_n-y_m)|^2dx\to0\mbox{ as }n,m\to\9.$$
This implies that the sequence
 $\{\sqrt{b^*(u)}\,\nabla y_n\}$ is convergent in $L^2$ as $n\to\9$ and by Hypothesis (ii) so is $\{\sqrt{u}\,\nabla y_n\}$. Let
\begin{equation}\label{e3.5}
f=\lim_{n\to\9}\sqrt{b^*(u)}\,\nabla y_n\mbox{ in }L^2.
\end{equation}
As $\frac\psi u\in L^1$ for some $\psi\in\calx$,   we infer   that $\{\nabla y_n\}$ is convergent in $L^1_{\rm loc}$ and so, by the Sobolev embedding theorem (see, e.g., \cite{5a}, p.~278), the sequence $\{y_n\}$ is convergent in $L^{\frac d{d-1}}_{\rm loc}$ and, therefore, in $L^1_{\rm loc}$ too. Hence, as $n\to\9$, we have
$$\barr{rcll}
y_n&\longrightarrow&y&\mbox{ in }L^1_{\rm loc}\cap L^{\frac d{d-1}}_{\rm loc},\vsp 
\nabla y_n&\longrightarrow&\nabla y&\mbox{ in }(L^1_{\rm loc})^d.\earr$$
and hence, along a subsequence, a.e. on $\rrd$. 
So, by \eqref{e3.5} we infer that $f=\sqrt{b^*(u)}\,\nabla y$, where $y\in  W^{1,1}_{\rm loc}$. Hence, as $n\to\9$, we have
$$\|z_n-z\|_u\to0\mbox{ for } z=-{\rm div}(b^*(u)\nabla y),\ y\in  W^{1,1}_{\rm loc},$$as claimed.  

As a consequence, we have that
\begin{equation}\label{e3.7}
\{z=-\divv(b^*(u)\nabla y);y\in C^\9_0(\rrd)\}\mbox{ is dense in }\calt_u(\calp^*)\mbox{ for all }u\in\calp^*.
\end{equation}

To conclude, we have shown that, {\it for each $u\in\calp^*$, $\calt_u(\calp^*)$ is a Hilbert space with the scalar product \eqref{e3.2}} and, as mentioned earlier, this is just the {\it tangent space} to $\calp^*$ at $u$. 

\section{The \FP\ gradient flow on $\calp^*$}
\label{s4}
\setcounter{equation}{0}
We are going to define here the gradient of the energy function $E:L^2\to]-\9,+\9]$ defined by \eqref{e1.10}. Namely, 
$$E(u)=\left\{
\barr{ll}
\dd\int_\rrd(\eta(u)+\Phi u)dx&\mbox{ if }u\in\calp\cap L^\9(\rrd)\cap L^1(\rrd;\Phi dx)\vsp
+\9&\mbox{ otherwise}.\earr
\right.$$
We note that $E$ is convex, nonidentically $+\9$ and we have:

\begin{lemma}\label{l4.1} $E$ is lower-semicontinuous on $L^2$.\end{lemma}

\begin{proof} We first note that if $u\in\calp\cap L^\9\cap L^1(\rrd;\Phi dx)$, then by the proof of (4.6) in \cite{4} for all $\alpha\in [m/(m+1),1)$,  we have by Hypothesis (iv)
	$$\int_\rrd \eta(u)dx\ge-C_\alpha\(\int_\rrd\Phi u\,dx+1\)^\alpha,$$
	hence, since $r^\alpha\le\frac1{2C_\alpha}\,r+C'_\alpha,\ r\ge0,$
\begin{equation}\label{4.0}
	E(u)\ge\frac12\int_\rrd\Phi\,u\,dx-C''_\alpha
\end{equation}
	for some $C_\alpha,C'_\alpha,C''_\alpha\in(0,\9)$ independent of $u$.
	
	Let now $u,u_n\in L^2$, $n\in\nn$, such that $\lim\limits_{n\to\9}u_n=u$ in $L^2$. We may assume that
	$$\liminf_{n\to\9}E(u_n)=\lim_{n\to\9}E(u_n)<\9$$
	and that $E(u_n)<\9$ for all $n\in\nn$. Then, by \eqref{4.0} 
	\begin{equation}\label{4.0'}
		\sup_{n\in\nn}\int_\rrd\Phi\,u_n\,dx<\9.
	\end{equation}
	Now, suppose that
	\begin{equation}\label{4.0''}
		E(u)>\lim_{n\to\9} E(u_n).
	\end{equation}
Then
$$\barr{ll}
E(u)\!\!\!&>\dd\liminf_{n\to\9}\int_\rrd\eta(u_n)dx+
\liminf_{n\to\9}\int_\rrd\Phi\,u_n\,dx\vsp
&\ge\dd\liminf_{n\to\9}\int_\rrd\eta(u_n)dx+\int_\rrd\Phi\,u\,dx,\earr$$
where we applied Fatou's lemma to the second summand in the last inequality. If we can also apply it to the first, then we get a contradiction to \eqref{4.0''} and the lemma is proved. To justify the application of Fatou's lemma to the first summand, it is enough to prove that there exist $f_n\in L^1$, $n\in\nn$, $f_n\ge0$, such that (along a subsequence)
\begin{equation}\label{4.0'''}
f_n\to f\mbox{ in }L^1,
\end{equation}
and
\begin{equation}\label{4.0'v}
\eta(u_n)\ge-f_n, n\in\nn. 
\end{equation}
To find such $f_n$, $n\in\nn$, we use \eqref{4.0'}. Recall from (4.4) in \cite{4} that, for some $c\in(0,\9)$,
$$\eta(r)\ge-cr\log^-(r)-cr,\ r\ge0.$$Hence,
$$\eta(u_n)\ge-cu_n\log^-(u_n)-cu_n,\ n\in\nn.$$
Since $u_n\to u$ in $L^2$ and thus in $L^1_{\rm loc}$, it follows by\eqref{4.0'} and our assumptions on $\Phi$ that (again along a subsequence) $u_n\to u$ in $L^1$. Furthermore, for all $\alpha\in(0,1)$,
$$\barr{ll}
-f\log^-(r)\!\!\!&=1_{[0,1]}(r)r\log r
=1_{[0,1]}(r)r\,\frac1{1-\alpha}\,r^\alpha
\underbrace{r^{1-\alpha}\log r^{1-\alpha}}_{\ge-e\1}\vspace*{-3mm}\\
&\ge-\dd\frac1{(1-\alpha)e}\,r^\alpha,\ r\ge0.\earr$$
Hence, we find that
$$\eta(u_n)\ge-\frac c{(1-\alpha)e}\,u^\alpha_n-c u_n,\ n\in\nn.$$
But, since $u_n\to u$ in $L^2$ and thus $u^\alpha_n\to u^\alpha$ in $L^1_{\rm loc}$, by Hypothesis (iv) it remains to show that, for some $\vp,\alpha\in(0,1)$,
\begin{equation}\label{4.0v}
	\sup_{n\in\nn}\int_\rrd u^\alpha_n\,\Phi^\vp\, dx<\9,
	\end{equation}
to conclude that (along a subsequence) $u^\alpha_n\to u^\alpha$ in $L^1$, and then (again selecting a subsequence of $\{u_n\}$ if necessary) \eqref{4.0'''} and \eqref{4.0'v} hold with
$$f_n:=\frac c{(1-\alpha)e}\,u^\alpha_n+cu_n,\ n\in\nn.$$So, let us prove \eqref{4.0v}.

Applying H\"older's inequality with $p:=\frac1\alpha$, we find that
$$\int_\rrd u^\alpha_n\,\Phi^\vp\, dx\le\(\int_\rrd u_n\Phi\,dx\)^\alpha\(\int_\rrd\Phi^{-\(\frac1\alpha-\vp\)/(1-\alpha)}\,dx\)^{1-\alpha}.$$

Hence, choosing $\vp$ small enough and $\alpha$ close enough to $1$, so that\break $\(\frac1\alpha-\vp\)/(1-\alpha)\ge m$, Hypothesis (iv) implies \eqref{4.0v}.\end{proof}

By Lemma \ref{l4.1} we have for $E$ that its directional derivative
$$E'(u,z)=\lim_{\lbb\to0}\frac1\lbb(E(u+\lbb z)-E(u))$$
exists for all $u\in\calp^*$ and $z\in L^2$ (it is unambiguously either a real number or $+\9$) (see, e.g., \cite{3a}, p. 86). 

In the following, we shall take $u\in\calp^*\subset D(E)=\{u\in L^2;E(u)<\9\}$ and $z\in\calt_u(\calp)$ and obtain 	that 
\begin{equation}\label{e4.1a}
	\barr{lcl}
E'(u,z)&=&\dd\lim_{\lbb\downarrow0}\frac1\lbb(E(u+\lbb z)-E(u))\\
&=&\dd\int_\rrd  z(x)\(\int^{u(x)}_1\frac {\beta'(\tau)}{b^*(\tau)}\,d\tau+\Phi(x)\)dx.\earr\end{equation}

Moreover, the subdifferential $\pp E_u:L^2\to L^2$ of $E$ at $u$ is expressed as (see \cite[Proposition 2.39]{3a})
\begin{equation}\label{e4.1b}
\pp E_u=\{y\in L^2;\ (z,y)_2\le E'(u,z);\ \ff z\in L^2\}.
\end{equation}
We recall that if $E$ is G\^ateaux differentiable at $u$, then $\pp E_u$ reduces to the gradient $\nabla E_u$ of $E$ at $u$ and 
$$E'(u,z)=(\nabla E_u,z)_2,\ \ff z\in L^2.$$
Any element $y\in\pp E_u$ is called a {\it subgradient} of $E$ at $u$. In the following, we shall denote, for simplicity,  again by $\nabla E_u$ any subgradient of $E$ at $u$ and we shall keep the notation ${\rm diff}\,E_u\cdot z$ for $E'(u,z)$.	

If $z\in\calt_u(\calp^*)$ is of the form $z=z_2=-\divv(b^*(u)\nabla y_2),$ where $y_2\in C^\9_0(\rrd),$ then $z=-b^*(u)\Delta y_2-\nabla y_2\cdot(b'(u)u+b(u))\nabla u$ and so, by (i) and \eqref{e1.10b}, it follows that $z\in L^2$ and  hence

\begin{equation}\label{e4.2}
\barr{lcl}
E'(u,z)&=&{\rm diff}\,E_u\cdot z =\dd\lim_{\lbb\downarrow0}\frac1\lbb(E(u+\lbb z_2)-E(u))\vsp
&=&\dd\int_\rrd\(\frac{\nabla\beta(u(x))}{b^*(u(x))}-D(x)\)b^*(u(x))\cdot\nabla y_2(x)\,dx\vsp
&=&\dd\int_\rrd b^*(u(x))\nabla y_2(x)\cdot\nabla\(\int^{u(x)}_0\frac{\beta'(s)}{b^*(s)}\,ds+\Phi(x)\)dx.\earr
\end{equation}
We claim that
\begin{equation}\label{4.3'}
x\mapsto\int^{u(x)}_0\frac{\beta'(s)}{b^*(s)}\,ds+\Phi(x)\mbox{\ \ is in }W^{1,1}_{\rm loc}.\end{equation}
To prove this, we first note that by Hypotheses (i) and (ii)
$$\frac{\gamma_1}{|b|_\9}\ \frac1s\le\frac{\beta'(s)}{b^*(s)}\le\frac{\gamma_2}{b_0}\ \frac1s,\ s>0.$$Hence,
\begin{equation}\label{4.3''}
\frac{\gamma_1}{|b|_\9}\log u\le\int^u_0\frac{\beta'(s)}{b^*(s)}\,ds\le\frac{\gamma_2}{b_0}\log u.\end{equation}
Now, let $\psi\in\calx$ such that $\frac\psi u\in L^1$. then, for every compact $K\subset\rrd$ and $K_n:=\left\{\frac1n\le u\le1\right\}$, $n\in\nn$,
$$\barr{ll}
\dd\int_{K_n}(\log u)^-dx\!\!\!
&\le\(\dd\int_{K_n}(\log u)^2u\,dx\)^{\frac12}
\(\inf_K\psi\)^{-\frac12}
\(\int_K\frac\psi u\,dx\)^{\frac12}\vsp 
&\le\dd\sup_K((\log u)^-u)
\(\int_{K_n}(\log u)^-dx\)^{\frac12}
\(\inf_k\psi\)^{-\frac12}
\(\int_K\frac\psi u\,dx\)^{\frac12}.
\earr$$
Dividing by $\(\int_{K_n}(\log u)^-dx\)^{\frac12}$ and letting $n\to\9$ yields $\log u\in L^1_{\rm loc}$, since trivially $(\log u)^+\in L^1_{\rm loc}$, since $u\in L^\9$. Furthermore, for $\vp>0$,
$$\barr{ll}
\dd\int_K|\nabla\log(u+\vp)|dx\!\!\!
&\dd\int_K\frac{|\nabla u|}{u+\vp}\,dx
\le\(\dd\int_K\frac{|\nabla u|^2}u\,dx\)^{\frac12}\(\inf_K\psi\)^{-\frac12}
\(\int_K\frac\psi u\,dx\)^{\frac12}.
\earr$$
Letting $\vp\to0$ yields $|\nabla \log u|\in L^1_{\rm loc}$, and \eqref{4.3'} is proved by Hypothesis (iv). Hence,
$$E'(u,z_2)=\dd\int_\rrd b^*(u(x))\nabla y_2(x)\cdot\nabla  y_1(x)dx
=-\<z_1,z_2\>_u,$$
where $z_1=-{\rm div}(b^*(u)\nabla y_1),$ $y_1=\int^u_0\frac{\beta'(s)}{b^*(s)}\,ds+\Phi.$ Therefore,
$$z_2\mapsto E'(u,z_2)={\rm diff}\,E_u z_2=\<\nabla E_u,z_2\>$$
extends to all $z\in\calt_u(\calp^*)$ by continuity and 
 by \eqref{e3.2} it follows that for $u\in\calp^*$  any subgradient $\nabla_uE$ of $E$ is given by 
\begin{equation}
\label{e4.3}
\barr{ll}
\nabla E_u\!\!\!&=-
{\rm  div}\(b^*(u)
\nabla\!\(\dd\int^u_0\!
\frac{\beta'(s)}{b^*(s)}\,ds+\Phi\)\)\vsp &
=-\Delta\beta(u)+
{\rm div}(Db^*(u))\in H\1.\earr\end{equation}
In particular, this means that $\pp E_u$ is single valued and $\pp E_u=\nabla E_u.$ 

On the other hand, by Theorem \ref{t1} we know that, for $u_0\in\calp^*$ with $u_0\log u_0\in L^1$, we have for the flow $u(t)\equiv S(t)u_0$,
$$S(t)u_0\in H^1\cap\calp,\ \ff t>0,\ \frac{\nabla (S(t)u_0)}{\sqrt{S(t)u_0}}\in L^2\mbox{, a.e. $t>0,$}$$ 
\begin{eqnarray}
&&\dd\frac{d^+}{dt}\,S(t)u_0=\Delta\beta(S(t)u_0)-\divv(Db^*(S(t)u_0)),\ \ff t>0,\label{e4.4}\\[1mm]
&&\dd\frac d{dt}\,S(t)u_0=\Delta\beta(S(t)u_0)
-\divv(Db^*(S(t)u_0)),\ \ff t\in(0,\9)\setminus N,\qquad\label{e4.5}
\end{eqnarray}
where $N$ is at most countable set of $(0,\9)$. Moreover, if   ${u_0}\in\calp^*$, then, as seen in Theorem \ref{t1}, it follows that $S(t)u_0\in\calp^*$, $\ff t>0$,  and   $\nabla E_{u(t)}$ is well defined, a.e. $t>0$. Taking into account \eqref{e4.3}, we may rewrite \eqref{e4.4}-\eqref{e4.5} as the gradient flow on $\calp^*$ endowed with the metric tensor \eqref{e3.2}. Namely, we have
\begin{theorem}\label{t3.1} Under Hypotheses {\rm(i)--(iv)}, for each $u_0 \in \calp^*$,  
	the function $u(t)=S(t)u_0\in\calp^*$, $\ff t>0$, and it is the solution to the gradient flow 
\begin{eqnarray}
\frac d{dt}\,u(t)&=&-\nabla E_{u(t)}, \mbox{ a.e. }t>0,\label{e4.6b}\\
\frac{d^+}{dt}\,u(t)&=&-\nabla E_{u(t)},\ \ff t>0,\label{e3.6}\\
\frac{d}{dt}\,u(t)&=&-\nabla E_{u(t)},\ \ff t\in(0,\9)\setminus N,\label{e3.7}
\end{eqnarray}
where $N$ is  at most countable set   of $0,\9)$. 
\end{theorem}

By \eqref{e3.2} we may rewrite \eqref{e3.6} as
\begin{equation}\label{e4.9a}
	\frac{d^+}{dt}\,E(S(t)u_0)=-\left\|\frac{d^+}{dt}\,S(t)u_0\right\|^2_{u(t)},\ \ff t>0.\end{equation}
 Equivalently,
\begin{equation}\label{e4.9}
	\frac{d^+}{dt}\,E(S(t)u_0)+A(S(t)u_0)=0,\ \ff t>0, 
	\end{equation}
where $A^*$ is the generator \eqref{e2.1} of the \FP\ semigroup $S^*(t)$ (equi\-va\-lently, $S(t)$) in $H\1$. Similarly, by \eqref{e3.2a} and \eqref{e2.17}--\eqref{e2.23a} we can write 
\begin{equation}\label{e4.10a}
\frac d{dt}\,E(S(t)u_0)=-\left\|\frac d{dt}\,S(t)u_0\right\|^2_{u(t)}=\Psi(S(t)u_0),\ \ff t\in(0,\9)\setminus N.\end{equation} 
As a matter of fact, the energy dissipation formula \eqref{e4.10a}  was used in \cite{4} (see also \cite{2}, Chapter 4) to prove that $S(t)u_0\to u_\9$ strongly in $L^1$ as \mbox{$t\to\9$}, where $u_\9$ is the unique solution to equilibrium equation $-\Delta\beta(u_\9)+{\rm div}(Db(u_\9)u_\9)=0$. 

\begin{remark}\rm Taking into account \eqref{e4.1a}, we see also that the operator $A^*$ defined by \eqref{e2.1} can be expressed as
	\begin{equation}\label{e4.7}
		A^*u=B_u\,{\rm diff}\,E_u,\ \ff u\in D(A^*)=H^1,
		\end{equation}
	where $B_u:H^1\to H\1$ is the linear symmetric operator defined by
	\begin{equation}\label{e4.8}
		\barr{rcl}
		B_u(y)&=&-{\rm div}(b^*(u)\nabla y),\ \ff y\in D(B_u),\vsp 
		D(B_u)&=&\{y\in l^2,\ \sqrt{u}\ \nabla {y}\in L^2\}.\earr
	\end{equation}
This means that $\nabla E_u$ can be equivalently written as
	\begin{equation}\label{e4.9}
	\nabla E_u=B_u({\rm diff}\ E_u). 
\end{equation}
\end{remark}

In the special case $b(r)\equiv1$,
$$E_u\equiv\int^u_1\frac{\beta'(\tau)}{\tau}\ d\tau+\Phi$$
and so $u(t)=S(t)u_0$ is the Wasserstein gradient flow of the functional $E$ defined by the time-discretized scheme
$$\barr{rcl}
u_h(t)&=&u^j_h,\ t\in[jh,(j+1)h),\ j=0,1,...,\vsp
u^j_h&=&\dd\min_u\left\{\frac1{2h}\ d_2(u,u^{j-1}_h)+E(u)\right\}\earr$$
where $d_2$ is the Wasserstein distance of order two (see \cite{0}, \cite{9aa}, \cite{9}). However, in the general case considered here, this is not the case.

\bk\n{\bf Acknowledgement.} This work was funded by the  DFG (German Research Foundation) Project ID 317210226-SFB 1283.

\end{document}